\documentclass[letterpaper, 10 pt, conference]{ieeeconf}  

\IEEEoverridecommandlockouts                              

\usepackage{graphics} 
\usepackage{epsfig} 
\usepackage{mathptmx} 
\usepackage{times} 
\usepackage{amsmath} 
\usepackage{amssymb}  
\usepackage{xcolor,soul}
\usepackage{relsize,exscale}
\usepackage[export]{adjustbox}

\DeclareMathAlphabet{\mathcal}{OMS}{cmsy}{m}{n}
\SetMathAlphabet{\mathcal}{bold}{OMS}{cmsy}{b}{n}

\title{\LARGE \bf
Fully decentralized conditions for local convergence of DC/AC converter network based on matching control}

\author{Taouba Jouini$^{1}$, Zhiyong Sun$^{2}$ 
	\thanks{*This work has received funding from the European Research Council (ERC) under the European Union's Horizon 2020 research  and innovation program (grant agreement No: 834142).}%
	\thanks{$^{1}$Taouba Jouini is with the Department of Automatic Control, LTH, Lund University,
		Ole Römers väg 1,  22363 Lund, Sweden. 
				$^{2}$Zhiyong Sun is with Department of Electrical Engineering,  Eindhoven University of Technology, the Netherlands.
		E-mails:
		\tt\small taouba.jouini@control.lth.se, z.sun@tue.nl.}}%

\graphicspath{{./fig/}}
\usepackage{epstopdf}

\usepackage{amsmath,amsthm,amssymb,mathrsfs,dsfont}
\usepackage{amsfonts}
\usepackage{siunitx}
\usepackage[colorinlistoftodos,prependcaption,textsize=tiny]{todonotes}
\usepackage{tikz}
\usetikzlibrary{matrix}

\newcommand\scalemath[2]{\scalebox{#1}{\mbox{\ensuremath{\displaystyle #2}}}}

\usepackage{cite}

\newcommand\oprocendsymbol{\hbox{$\blacksquare$}}

\newcommand\oprocend{\relax\ifmmode\else\unskip\hfill\fi\oprocendsymbol}

\newcommand{\real}[0]{\mathbb R}

\newtheorem{theorem}{Theorem}[section]
\newtheorem{lemma}[theorem]{Lemma}
\newtheorem{definition}[theorem]{Definition}

\newtheorem{assumption}[]{Assumption}

\newtheorem{condition}[theorem]{Condition}

\newcommand{\tb}[0]{\color{blue}}

\usepackage[normalem]{ulem}
\newcommand\rout{\bgroup\markoverwith{\textcolor{red}{/}}\ULon} 

\usepackage[prependcaption,colorinlistoftodos]{todonotes}

\begin{document}

\maketitle
\thispagestyle{empty}
\pagestyle{empty}

\begin{abstract}
We investigate local convergence of identical DC/AC converters interconnected via identical resistive and inductive lines towards a synchronous equilibrium manifold. We exploit the symmetry of the resulting vector field and develop a Lyapunov-based framework, in which we measure the distance of the solutions of the nonlinear power system model to the equilibrium manifold by analyzing the evolution of their tangent vectors. We derive sufficient and fully decentralized conditions to characterize the  equilibria of interest, and provide an estimate of their region of contraction. We provide ways to satisfy these conditions and illustrate our results based on numerical simulations of a two-converter benchmark.
\end{abstract}


\section{INTRODUCTION}
In the advent of high penetration of renewable energy resources in the electrical network \cite{farrokhabadi2019microgrid}, power system stability remains at the heart of the understanding of the ramifications of these unprecedented changes affecting the generation, operation and distribution of energy, where power electronic DC/AC converters, play key role in maintaining reliable power supply.
\paragraph{Literature review}
Despite their intrinsic differences, synchronous machines and DC/AC converters share structural similarities, which are often exploited to design efficient control strategies that endow resilience to the electrical grid. Thus, different schemes for machine emulating control e.g., droop control \cite{simpson2013synchronization}, virtual synchronous machines \cite{bevrani2014virtual}, synchronverters \cite{Z11}, are extensively studied and labeled as {\em grid-friendly} for meeting power demands and showing robustness against common disturbances. In particular, the matching control, introduced recently in \cite{jouini2016grid} has gained much attention, due to its simple implementation and advantageous plug and play properties \cite{arghir2018grid}.

One of the major difficulties in the network analysis of power system stability, is the presence of a continuum of steady states due to the symmetry of the vector field describing the multi-converter or multi-machine dynamics \cite{G18}. In particular, the rotational invariance indicates the absence of a reference frame or absolute angle in power system and presents a fundamental obstacle for defining suitable error coordinates. A common approach in power system literature is to eliminate this continuum of equilibria, e.g., by performing transformations either resulting from grounding a node \cite{tegling2015price}, or projecting into the orthogonal complement  \cite{schiffer2019global}, if the equilibrium manifold is a linear subspace, where classical stability tools such as Lyapunov direct method, can be deployed. Nonetheless, this type of transformations are not possible for high-order systems, where the dynamics do not have a direct coupling term or if the Laplacian matrix cannot be expressed explicitly. This has been highlighted at different occasions in power system literature.

Differential geometric methods have been adopted in the study of nonlinear solutions of symmetric vector fields and we distinguish two main avenues. First, {\em contraction theory and differentiation methods} \cite{lohmiller1998contraction} assess the stability of nonlinear trajectories, in terms of their convergence with respect to one another. Contraction theory captures the convergence towards a particular solution with a specific smooth property \cite{russo2011symmetries} relying on infinitesimal virtual displacements. The study of differential system dynamics on the tangent bundle shows for example the convergence to an {\em attractor} for coupled identical nonlinear oscillators \cite{wang2005partial}. Contracting systems are also referred to as {\em convergent}.

Second, {\em Lyapunov theory and incremental methods} \cite{angeli2002lyapunov} have been recognized as promising tools to study the stability of trajectories with respect to one another, besides being attracted towards an equilibrium of interest. Incremental Lyapunov theory is tailored to power system models in the aftermath of a failure or disturbance from an energy-shaping perspective. 
 
Another approach, under the name of {\em differentiable Lyapunov framework}{\tb,} merges integration methods revolving around incremental Lyapunov functions with differentiation methods, based on contraction analysis \cite{forni2013differential}. This approach allows the study of stability of nonlinear trajectories by looking at the dynamics of their virtual displacements and measures well-defined distance, called Finsler distance, between them via integration.

\paragraph{Contributions}
In this work, we consider high-fidelity power system model, consisting of identical DC/AC converters interconnected via identical resistive and inductive lines. Based on preliminary results in \cite{jouini2019parametric}, we exploit structural properties of the vector field to prove convergence of the nonlinear trajectories, under fully decentralized conditions, which can be verified individually at each converter. 
For this, we adopt the differential stability framework presented in \cite{forni2013differential}, by lifting the Lyapunov function to the tangent bundle. Based on considerations in the quotient manifold, we show that solutions of the multi-converter system converge towards a synchronous equilibrium manifold on a contraction region characterized by small distance of the angles, frequency and AC signals to the subspace, representing the tangent vector of the rotational invariance at steady state. We link our stability results theories in the study of weak/partially contracting systems. Our simulations illustrate our results, where the stability conditions are satisfied and the contraction region is numerically estimated.

\paragraph{Notation}:
We define an undirected graph $G=(\mathcal{V}, \mathcal{E})$, where $\mathcal{V}$ is the set of nodes with $\vert\mathcal{V}\vert =n$ and $\mathcal{E}\subseteq \mathcal{V}\times \mathcal{V}$ is the set of interconnected edges with $\vert \mathcal{E}\vert=m$. We assume that the topology specified by $\mathcal{E}$ is arbitrary and define the map $\mathcal{E} \to \mathcal{V}$, which associates each oriented edge $e_{ij}=(i,j) \in \mathcal{E}$ to an element from the subset $ \mathcal{I}=\{-1,0,1\}^{|\mathcal{V}|}$,  
resulting in the incidence matrix $\mathcal{B}\in\mathbb{R}^{n\times m}$.
We denote by $\mathds{1}_n$ the vector of all ones,  $I\in\real^{2\times 2}$ the identity matrix $I=\left[\begin{smallmatrix}
1& 0\\ 0& 1
\end{smallmatrix}\right]$, $\mathbf{I}$ the identity matrix of dimensions $p$, with $p\in\mathbb{N}$ and $\mathbf{J}= \mathbf{I} \otimes J$ with $J=\left[\begin{smallmatrix}
0& -1\\ 1& 0
\end{smallmatrix}\right]$.
We define the rotation matrix $R(\gamma)=\left[\begin{smallmatrix}
\cos(\gamma) & -\sin(\gamma) \\ \sin(\gamma) &\cos(\gamma) 
\end{smallmatrix}\right]$, 
and $\textbf{R}(\gamma)= \mathbf{I}\otimes R(\gamma)$. Let $\textrm{diag}(v)$ denote a diagonal matrix, whose diagonals are elements of the vector $v$ and $\mathrm{Rot}(\gamma)=\text{diag}(r(\gamma_{1}), \dots r(\gamma_{n})),\; k=1\dots n$, with $r(\gamma_k)=\begin{bmatrix} -\sin(\gamma_k) & \cos(\gamma_k)\end{bmatrix}^\top$. Let $\mathbb{S}^1$ be the unit circle, and $\mathbb{T}^n=\mathbb{S}^1\times \dots \mathbb{S}^1$ the $n$-th dimensional torus. Let
$d(\cdot, \cdot)$ be the distance metric. 
Given a set $\mathcal{A}\subseteq\mathds{R}^n$, we denote by $T_z\,\mathcal{A}$ the tangent space of $\mathcal{A}$ at $z$ and the tangent bundle of $\mathcal{A}$ by $T\mathcal{A}=\bigcup\limits_{z\in\mathcal{A}} \{z\}\times T_z\mathcal{A}$. Let $\mathcal{K}$ and $\mathcal{K}_{\infty}$ be comparison functions  defined by all the maps $k:\real_{\geq 0}\to \real_{\geq 0}$, that are continuous and strictly increasing, where $k(0)=0$. For $\mathcal{K}_{\infty}$ functions, it holds that $k(t)\to \infty$ as ${t\to\infty}$.

\section{System setup}
\label{sec: setup}
\subsection{Multi-source power system dynamics}
\label{subsec: model}
We start from the following general model describing the evolution of identical DC/AC converters in closed-loop with the matching control \cite{arghir2018grid}, a control strategy that renders the closed-loop DC/AC converter structurally similar to a synchronous machine, interconnected with identical resistive and inductive lines.

We model the dynamics of a balanced and averaged three-phase DC/AC converter in closed-loop with matching control, after transformation into a rotating $dq$ frame, at the nominal steady state frequency $\omega_n>0$, with angle $\theta_{dq}(t)=\int_0^t \omega_n \,d\tau$ (by the so-called Clark transformation~\cite{K94}), given by first-order differential equations,
\begin{align}
\begin{bmatrix} \dot {\gamma}_k  \\C_{dc} \dot v_{dc,k} \\ L\dot {i}_{k} \\ C\dot {v}_{k}   \end{bmatrix}=  \left[\begin{smallmatrix}
\eta (v_{dc,k}-v_{dc}^*) \\ -K_p (v_{dc,k}-v_{dc}^*)-\frac{\mu}{2} r(\gamma_{k})^{\top}i_{k}\\
-(R\, I+L\, \omega_n\, J )\, i_{k}+ \frac{\mu}{2} r(\gamma_{k}) v_{dc{,k}}-v_{k}\\
 -(G\, I+C\, \omega_n\,J)\, v_{k} +i_{k}-i_{net,k}
\end{smallmatrix}\right]+ \begin{bmatrix}
0 \\ i^*_{dc,k}\\ 0 \\ 0  
\end{bmatrix}
\label{eq: i-node}
\end{align}
where $\gamma_k\in \mathbb{S}^1$ is the virtual converter angle, $\eta$ is a positive control gain, $\dot \gamma_k=\omega_k\in \mathbb{R}$ is the relative (to the nominal) frequency. Let $v_{dc,k}\in\real$ denote the DC voltage across the DC capacitor with nominal value $v_{dc}^*$. The parameter $C_{dc}>0$ represents the DC capacitance and the  conductance $G_{dc}>~0$ together with the proportional control gain $\hat K_p>0$, are represented by $K_p=G_{dc}+\hat K_p$. This results from designing a controllable current source $ i_{dc,k}=\hat K_p(v_{dc,k}-v_{dc}^*)+i^*_{dc,k}$, where we denote by $i^*_{dc,k}\in\real$ a constant current source representing DC side input to the converter. Let $\mu\in[0,1]$ be the constant modulation amplitude,
$i_{k} \in  \mathbb{R}^2$ the inductance current, $v_{k} \in  \mathbb{R}^2$ the output voltage and $i_{\ell,k}\in  \mathbb{R}^2$ the line current. 
The filter resistance and inductance are represented by $R>0$ and $L>0$. The capacitor $C>0$ is set in parallel with the load conductance $G>0$ to ground and connected to the network via the line current $i_{net,k}\in\real^{2}$.

 
By lumping the states of $n$ identical converters and $m$  identical  lines and defining the impedance matrices $Z_R=R\; \mathbf{I}+L\,\omega_n\, \mathbf{J}, \, Z_C=G\;  \mathbf{I}+C\,\omega_n\,\mathbf{J},\, Z_\ell=R_\ell \; \mathbf{I}+L_\ell\omega_n\, \mathbf{J}$, we obtain the following power system model, 
\begin{align}
\begin{bmatrix} \dot {\gamma}  \\\dot v_{dc} \\ \dot {i}_{} \\ \dot {v}_{}  \\ \dot {i}_{\ell} \end{bmatrix}= K^{-1} \left[\begin{smallmatrix}
\eta (v_{dc}-v_{dc}^*\mathds{1}_n)\\ -K_p (v_{dc}-v^*_{dc}\mathds{1}_n)-\frac{\mu}{2} \mathrm{Rot}(\gamma)^{\top}\, i_{}\\
-Z_R\,  i_{}+ \frac{\mu}{2} \mathrm{Rot}(\gamma) \, v_{dc}-v_{}\\
-Z_C\, {{v}}_{}-\mathbf{B}\, i_\ell+i \\
-Z_\ell\, i_{\ell}+\mathbf{B}^\top\, v_{}
\end{smallmatrix}\right]+K^{-1} \begin{bmatrix}
0 \\ \textbf{u}\\ 0 \\ 0 \\ 0 
\end{bmatrix}\,,
\label{eq: multi-node}
\end{align}
where we define the angle vector $\mathbf{\gamma} =\begin{bmatrix}
\gamma_1 \dots \gamma_n
\end{bmatrix}^{\top}\in \mathbb{T}^n$, with relative frequencies ${\omega} =\begin{bmatrix}
\omega_{1} \dots \omega_{n}
\end{bmatrix}^{\top}\in \mathbb{R}^n$,  DC voltage vector ${v_{dc}} =\begin{bmatrix}
v_{dc,1} \dots v_{dc,n}
\end{bmatrix}^{\top}\in \mathbb{R}^n$,  AC inductance current ${i} =\begin{bmatrix}
i_{1} \dots i_{n}
\end{bmatrix}^{\top}\in \mathbb{R}^{2n}$ and output capacitor voltage ${v} =\begin{bmatrix}
v_1 \dots v_n
\end{bmatrix}^{\top}\in \mathbb{R}^{2n}$. 
The last equation in \eqref{eq: multi-node} describes the line dynamics and in particular, the evolution of the line current ${i_\ell}:=\begin{bmatrix}
i_{\ell_1} \dots i_{\ell_m}
\end{bmatrix}^\top \in \mathbb{R}^{2m}$, where $R_{\ell}>0$ is the line resistance, $L_{\ell}>0$ is the line inductance,
$\mathbf{B}=\mathcal{B}\otimes {I}$ and $K=\textrm{diag}\left(\mathbf{I}, C_{dc}\; \mathbf{I}, L\; \mathbf{I} ,C\; \mathbf{I},L_\ell\;\mathbf{I} \right)$.
The multi-converter input is represented by $\mathbf{u}=\begin{bmatrix}
i^*_{dc,1}, \dots ,i^*_{dc,n}
\end{bmatrix}^\top\in\mathbb{R}^n$.

Let $N$ be the dimension of the state vector $z=\begin{bmatrix}
\gamma^\top & \tilde v_{dc} ^\top& x^\top
\end{bmatrix}^{\top}$. We define the relative DC voltage $ \tilde {v}_{dc}= v_{dc}-v_{dc}^*\mathds{1}_n$, the vector of AC signals $x=\begin{bmatrix}
i^\top_{} & v^\top_{} & i^\top_\ell
\end{bmatrix}^{\top}$ and the input $u=\begin{bmatrix} 0 & \textbf{u} & 0 & \dots & 0\end{bmatrix}\in~\real^N$ given by the vector in \eqref{eq: multi-node}. 

By putting it all together, we arrive at the nonlinear power system model compactly described by,
\begin{align}
\dot z={f}(z,  u),\;  
\label{eq: non-lin}
\end{align}
for all $z\in\real^N$, where ${f}(z,u)$ denotes the vector field in \eqref{eq: multi-node}.

Consider the nonlinear power system model  in \eqref{eq: non-lin}, for all $\theta\in\mathbb{S}^1$, it holds that,
\begin{align}
 f(\theta \, h_0+H(\theta)\, z,u)=f([z],u)= H(\theta)\, f(z,u)\,,
\label{eq:invariance}
\end{align}
where $ h_0=\begin{bmatrix}
\mathds{1}^\top_n & 0^\top& 0^\top 
\end{bmatrix}^\top, \ H(\theta)=\left[\begin{smallmatrix}
\mathbf{I} & 0 & 0 \\ 0 & \mathbf{I} & 0\\ 0 & 0 & \mathbf{R}(\theta)
\end{smallmatrix}\right],\;$ and 
\begin{align}
\label{eq: eq-class}
[z]=\left\{\begin{bmatrix}
(\gamma +\theta \mathds{1}_n)^\top &  \tilde v_{dc}^\top & (\mathbf{R}(\theta)\,x)^\top \end{bmatrix}^\top, \, \theta\in\mathbb{S}^1 \right\}.
\end{align} 
In fact, the rotation matrix $\mathbf{R}(\theta)$, commutes with the impedance matrices $Z_R,\,Z_C,\,Z_{\ell}$, the skew-symmetric matrix $\mathbf{J}$ and the incidence matrix $\mathbf{B}$. Notice that for $\theta=2\, k\,\pi,\, k\in\mathbb{Z}$, it holds that  $[z]=z$ and hence $z\in[z]$.

The symmetry \eqref{eq:invariance} arises from the fact that nonlinear power system model \eqref{eq: non-lin} has no absolute angle. In fact, a shift in all (virtual) angles $\gamma\in\mathds{T}^n$, induces a rotation in the angles of AC signals. Up to re-defining the $dq$ transformation angle to $\theta'_{dq}(t)=\theta_{dq}(t)+\theta$, the vector field \eqref{eq: non-lin} remains invariant under the translation and rotation actions in \eqref{eq:invariance}.

\subsection{Steady state manifold}
\label{sec: ss-manifold}
In light of Section \ref{subsec: model}, we aim to understand the properties of the induced synchronous equilibrium manifold,
\begin{align}
\mathcal{M}=\left.\{z^*\in \real^N \, \vert \: {f}(z^*,u)=0 \right.\},
\label{eq:ss-manifold}
\end{align}
resulting from setting \eqref{eq: non-lin} to zero, for a given input vector ${u}\in\real^N$ to be specified. Next, we investigate the properties of the equilibrium manifold $\mathcal{M}$ and define properties related to its symmetry and feasibility.
\begin{lemma}
\label{lem:ss-prop}
Consider the equilibrium manifold $\mathcal{M}$ described by \eqref{eq:ss-manifold}. Then, $\mathcal{M}$ has the following properties:
\begin{enumerate}
	\item {\underline {Synchronization}:} The frequencies of all converters synchronize at the nominal frequency $\omega_n$.
	\item {\underline {Rotational symmetry}}: $\mathcal{M}$ has a rotational symmetry given by the equilibrium manifold, 
	\begin{align}
   \!\!\!\!\!\!\!\!\!\!\![z^*]\!\!=\!\!\left\{\begin{bmatrix}
	(\gamma^*+\theta\mathds{1}_n)^\top & \! 0^\top \!&(\mathbf{R}(\theta)x^*)^\top
	\end{bmatrix}^\top,\,\theta\in\mathbb{S}^1\right\}, 
	\label{eq:rot-symm}
	\end{align}
	 that is, for all $z^* \in \mathcal{M}$, it holds that $[z^*] \in \mathcal{M}$.
	\item {\underline {Feasibility}}: If $\mathbf{u}(z^{*})=\frac{\mu}{2}\mathrm{Rot}(\gamma^*)^\top i^*_{}$, then $\mathcal{M}$ is non-empty.
\end{enumerate}

\end{lemma}

\begin{proof}
	$\empty$
	\begin{enumerate}
	\item By the chosen $dq$ frame with $\dot\theta_{dq}=\omega_n$, we set the angle dynamics in \eqref{eq: multi-node} to zero. This implies that for $v_{dc}=v_{dc}^*$, the relative frequency $\dot\gamma=\omega=0$ at steady state.
	\item The existence of a symmetry for $\mathcal{M}$ is a direct consequence of \eqref{eq:invariance}, applied to the steady state equations, satisfying~\eqref{eq:ss-manifold}.
	\item The feasibility condition follows from setting DC voltage dynamics in \eqref{eq: multi-node} to zero and solving for the input $i^*_{dc,k}$,  given by $i^*_{dc,k}-\frac{\mu}{2}{r}^\top(\gamma_k^*)\,i^*_{k}=0$, for $k=1\dots n$.
	
	\end{enumerate}
\end{proof}
\begin{assumption}[Feasibility of the steady states]
	Assume that the input $u$ in \eqref{eq: multi-node} is given by $\mathbf{u}=\mathbf{u}(z^*)$.
	\label{ass: feasibility}
\end{assumption}


\section{Local asymptotic contraction of power system model}
\label{sec: local-contraction}

\subsection{Preliminaries}
Under Assumption \ref{ass: feasibility}, we consider the power system model~\eqref{eq: non-lin}. Let $\mathcal{M}\subset\mathbb{R}^N$ be a non-empty steady-state manifold as defined in \eqref{eq:ss-manifold}. 
Because of the symmetry \eqref{eq:invariance}, we define the quotient manifold $\real^N/ \sim$ induced by the following equivalence relation for $ z_1=\begin{bmatrix}
\gamma_1^\top &\tilde v_{dc,1}^\top & x_1^\top
\end{bmatrix}^\top,\; z_2 =\begin{bmatrix}
\gamma_2^\top & \tilde v_{dc,2}^\top & x_2^\top
\end{bmatrix}^\top$, given by,
\begin{align}
\label{eq: equi-class}
z_1  \sim  z_2 & \, \text{ iff } \exists\, \theta \in \real,   \gamma_1-\gamma_2=\theta \mathds{1}_n,\;   x_1=R(\theta)\, x_2,
\end{align}
and defined by the equivalence class  \eqref{eq: eq-class}.
 The equivalence between two AC signals $x_1$ and $x_2$ follows from re-defining $dq$ frame angle.
Hence \eqref{eq: non-lin} represents a {\em quotient system} on $\real^N/ \sim$, in the sense of \cite[Sec. VIII-B]{forni2013differential} and \cite[Sec. B]{russo2011symmetries}: For every initial condition $z'_0\in[z_0]$, the solution $\phi (\cdot, z'_0)$ to  \eqref{eq: non-lin} satisfies $\phi (\cdot, z'_0)\in[\phi (\cdot, z_0)]$.

\begin{assumption}[Isolated equilibria on $\real^N/ \sim$]
\label{ass:isol-eq}
Consider the system  \eqref{eq: non-lin} defined on the quotient manifold $\real^N/ \sim$. Assume that the equilibria of the manifold $\mathcal{M}$  on $\real^N/ \sim$ are isolated.
\end{assumption}

Based on Assumption \ref{ass:isol-eq}, let $\mathcal{D}\subset \real^N$ be a neighborhood of $[z^*]\in\mathcal{M}$. In general, the steady state manifold $\mathcal{M}$ has multiple equilibria that are isolated (by Assumption \ref{ass:isol-eq}) on $\real^N/\sim$. We refer to the study of contraction of solutions of \eqref{eq: non-lin} restricted to a region of the space $\mathcal{D}$ containing $[z^*]\subset\mathcal{M}$ by {\em local} contraction analysis.

Next, we consider the following variational system on $\mathcal{D}$, (and implicitly on $\mathcal{D}/ \sim$),
\begin{align}
\label{eq:var-sys}
\dot { z}&=f(z,u), \\
\delta {\dot {z}}&=\frac{\partial f( z)}{\partial  z}\, {\delta z} \nonumber,
\end{align}
where ${\partial f(z)}/{\partial z}$ denotes the partial derivatives of \eqref{eq: non-lin} representing the Jacobian and $\delta z$ lies on $T_{z} \mathcal{D}$ the tangent space of $\mathcal{D}$ at $z$.

\begin{definition}[Lyapunov function with respect to $\mathcal{S}$]
	\label{def: LF}
	A differentiable function $V:\mathcal{U}\to \mathbb{R}, \: \mathcal{U} \subseteq \real^N$, is a Lyapunov function  with respect to a non-empty, closed and invariant set $\mathcal{S}\subseteq \mathcal{U}$, if  
	\begin{enumerate} 
		\item $V$ is positive definite with respect to $\mathcal{S}$, that is,
		\begin{itemize}
			\item $V(z)=0, \quad z \in \mathcal{S}$,
			\item $V(z)>0, \quad z \in \mathcal{U}\setminus\mathcal{S}$.
		\end{itemize}
		\item  Lie derivative of $V$ is negative definite with respect to $\mathcal{S}$, that is, 
		\begin{itemize}
			\item $\dot V(z)=0, \quad z \in \mathcal{S}$,
			\item $\dot V(z)<0, \quad z \in \mathcal{U} \setminus\mathcal{S}$.
		\end{itemize}
	\end{enumerate}
\end{definition}

Definition \ref{def: LF} is equivalent to the notion of  smooth Lyapunov function with respect to $\mathcal{S}$ using $\mathcal{K}_{\infty}$ functions introduced in \cite{lin1996smooth}.

Our analysis  of the Jacobian of the nonlinear power system model \eqref{eq: non-lin} in \cite{jouini2019parametric, jouini2019local}, takes under the loop the behavior of the differential system in \eqref{eq:var-sys} restricted to the tangent space  $T_{z^*}\mathcal{M}$ with $z^*\in \mathcal{M}$, as shown in Figure \ref{fig: old-analysis} and described by,
\begin{align}
\label{eq: lin-sys}
\delta {\dot z}=A(z^*) \; \delta z, \quad A(z^*)=\left[
\begin{array}{c|c}
A_{11} & A_{12} \\
\hline
A_{21} & A_{22}
\end{array}
\right],
\end{align}
with $\delta z=\begin{bmatrix}\delta z^\top_1 & \delta z^\top_2\end{bmatrix}^\top \in T_{z^*}\mathcal{M}$ corresponding to the partition $\delta z_1=\begin{bmatrix} \delta \gamma^\top & \delta \tilde v_{dc}^\top \end{bmatrix}^\top,\, \delta z_2=\delta x$. Note that the Jacobian ${A}(z^*)=\frac{\partial {f}}{\partial z}\vert_{z=z^*}$ has a one-dimensional zero subspace denoted by,  
 $$\mathrm{span}\{v(z^*)\}=\mathrm{span}\{\left[ \begin{array}{c c c} 
\mathds{1}_n {^\top} & 0^\top & (\mathbf{J} \, x^*)^\top \end{array}\right]^\top\},$$ where $\mathbf{J} \, x^*=\begin{bmatrix}(\mathbf{J}\,i_{}^*){^\top} & (\mathbf{J}v_{}^*){^\top} & (\mathbf{J}\, i_{\ell}^*){^\top}\end{bmatrix}^\top$.  For all $\delta z\in~T_{z^*}\mathcal{M}$, we show in \cite{jouini2019parametric} asymptotic stability of $\mathrm{span}\{v(z^*)\}$, in the sense of \cite[Theorem 2.8]{lin1996smooth} under the following steady state condition:

\begin{condition}[Equilibira of interest \cite{jouini2019parametric}]
	\label{cond:hurwitz-cdt-dcac}
	Consider a steady state $z^*\in\mathcal{M}$. Assume the following condition is satisfied at the $k$-th converter, 
	\begin{align}
	\label{eq: are-cdt}
	Q^*_{sw,k}& > \frac{\mu^2 v_{dc}^{*2}}{16 \,R},\; k=1\dots n,
	\end{align}
	where $Q_{sw,k}^*=\frac{1}{2}\, \mu (\mathbf{J}\, r(\gamma_k^*))^\top i_{k}^* \,v_{dc}^*$, denotes steady state reactive power after the switching block (before the output filter) at the $k$-th converter.
\end{condition}

\begin{figure}[h!]
	\centering
	\includegraphics[scale=0.7]{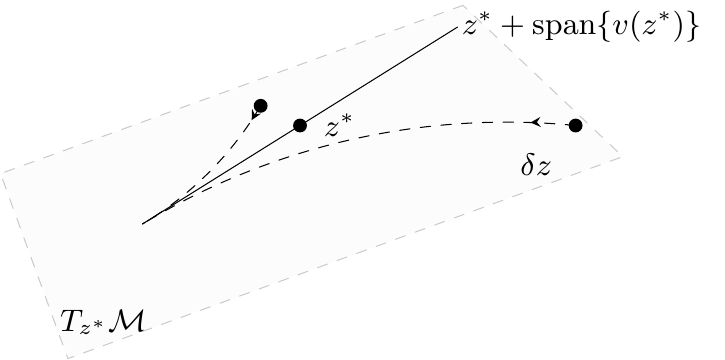}
	\caption{Proof of asymptotic stability of $\mathrm{span}\{v(z^*)\}$ in the sense of \cite{lin1996smooth} for trajectories on the tangent space  $T_{z^*}\mathcal{M}$ for the linearized power system model \eqref{eq: lin-sys} in \cite{jouini2019parametric}.}
	\label{fig: old-analysis}
\end{figure}

\begin{definition}[Finsler-Lyapunov function \cite{forni2013differential}]
	\label{def: LF}
	A differentiable function $V:T\mathcal{D}\times \real_{\geq 0}\to \mathbb{R}_{\geq 0}$, is a Finsler-Lyapunov function, if it satisfies 
	\begin{align}
	\label{eq: finsler-dist}
	   c_1\, \mathcal{F}(x,\delta x, t)^p \leq V(x,\delta x, t)\leq c_2\, \mathcal{F}(x,\delta x, t)^p,
	\end{align}
	for some $c_1, c_2>0$ and with $p\in\mathds{N}$, where $\mathcal{F}(x,\delta x, t)$ is a Finsler structure (see \cite{forni2013differential}), uniformly in $x$ and $t$.
\end{definition}

By the key property \eqref{eq: finsler-dist}, there exists a well-defined distance on $\mathcal{D}$ via integration defined below,
\begin{definition}[Finsler distance \cite{forni2013differential}]
\label{def: fins-dist}
Consider a candidate Finsler-Lyapunov function V on the manifold $\mathcal{X}$ and the associated Finsler structure $\mathcal{F}$. For any subset $ {X}\subset\mathcal{X}$, and any two points $z_1,z_2\in\mathcal{X}$, let $\Gamma(z_1,z_2)$ be the collection of piecewise $C^1$ curves, $\gamma: I \to \mathcal{X}$, connecting $z_1$ and $z_2$ with $\gamma(0)=z_1$ and $\gamma(1)=z_2$. The Finsler distance $d:\mathcal{X}\times \mathcal{X}\to\real_{\geq 0}$ induced by the structure $\mathcal{F}$ is defined by \begin{align}
d(z_1,z_2):= \inf_{\Gamma(z_1,z_2)}\int_\gamma \mathcal{F}\left(\gamma(s), \frac{\partial \gamma}{\partial s}, t\right)ds
\end{align}
\end{definition}
The pseudo-distance induced by $\mathcal{F} = \sqrt{V}$ on $\mathcal{D}$ is a distance on the quotient manifold $\mathcal{D}/\sim$.

To analyze the behavior of the linearized trajectories on the tangent bundle $T \mathcal{D}$ of the variational system \eqref{eq:var-sys}, we define a parameterized Lyapunov function $V:T\mathcal{D} \to\real$, with respect to $\mathcal{S}=\text{span}\{v(z^*)\}$ (from Definition \ref{def: LF} and as in \cite{jouini2019parametric}) and given by,
\begin{align}
\label{eq: LF}
\!\!\!\!\! V(\delta z)\!=\!\delta  z^\top \underbrace{\left(P-\frac{P\, v(z^*)v(z^*)^\top P}{v(z^*)^\top P \, v(z^*)} \right)}_{\Pi}\, \delta  z, \ P=\left[ \scalemath{0.7}{
\begin{array}{c|c} P_1 & 0\\ \hline 0 & P_2
\end{array}}\right] \!\! \!\! 
\end{align}
where $P$ is a symmetric, positive definite matrix with block diagonals $P_1>0$ and $P_2>0$ and $\delta z \in~T_z \mathcal{D}$. 
The Lyapunov function in \eqref{eq: LF} represents the squared distance of the tangent vector $\delta z\in\real^N$ to the linear subspace $\mathrm{span}\{v(z^*)\}$, in the weighted inner product defined by $\langle\cdot ,\cdot\rangle_P=(\cdot)^{\top}P\,(\cdot),\, P>0$ and the weighted Euclidean norm $\vert\vert \cdot \vert\vert_P=\sqrt{\langle \cdot,\cdot\rangle_P}$.
Note that by Definitions \ref{def: fins-dist} and \ref{def: LF}, $V(\delta z)$ represents a Finsler Lyapunov function.


We say that a solution to \eqref{eq:var-sys} is {\em  asymptotically contracting} on a forward invariant set $\mathcal{C}$, if for all initial conditions $z_1, z_2\in\mathcal{C}$, 
\begin{align}
\label{eq:increm-stab}
d\left(\phi(t,z_1), \phi(t,z_2)\right)&\leq k\, (d(z_1,z_2)) \\
\lim_{t\to\infty} d\left(\phi(t,z_1), \phi(t,z_2)\right)&=0, \nonumber
\end{align}
where $t>0$, $k(\cdot)$ is a $\mathcal{K}$-function and $d(\cdot,\cdot)$ is a pseudo-distance metric. Notice that if
$d(z_1,z_2)=0$, for all   $[z_1]=[z_2]$, then $d(\cdot, \cdot)$ becomes a distance metric on $\mathcal{C}/ \sim$ and the solutions to \eqref{eq:var-sys} satisfying \eqref{eq:increm-stab} on $\mathcal{C}/ \sim$  are  {\em incrementally asymptotically stable}, see \cite[Theorem 3]{forni2013differential}. 

It is noteworthy that the pseudo-distance induced by $\sqrt{V}$ on $\mathcal{D}$ in \eqref{eq: LF} is a distance on the quotient space $\mathcal{D}/\sim$.

\subsection{Local contraction analysis}

Since asymptotic contraction of \eqref{eq:var-sys}  on a forward invariant set $\mathcal{C}\subseteq \mathcal{D}$ is equivalent to incremental asymptotic stability on the quotient $\mathcal{C}/ \sim$, in the next section, we show incremental asymptotic stability of the quotient system \eqref{eq:var-sys} and characterize a forward invariant set $\mathcal{C}_\epsilon$. 

 \begin{theorem}
\label{thm: main}
Let the power system model \eqref{eq: non-lin}, under Assumption \ref{ass: feasibility}, \ref{ass:isol-eq} and Condition \ref{cond:hurwitz-cdt-dcac} be defined on a neighborhood $\mathcal{D}\subset \real^N$ of $[z^*]\subset\mathcal{M}$. Then, the solutions to~\eqref{eq: non-lin} asymptotically contract towards the synchronous equilibrium manifold $[z^*]$ on $\mathcal{C}_\epsilon\subseteq \mathcal{D}$, where,
\begin{align}
\label{eq: roc}
\mathcal{C}_\epsilon&=\left\{(z, \delta z)\in\mathcal{D}, V(\delta z)\leq \epsilon  \right\}, 
\end{align} 
with $\epsilon$ being positive and sufficiently small.
\end{theorem}

\begin{proof}
We consider the variational system \eqref{eq:var-sys} under Assumptions \ref{ass: feasibility}, \ref{ass:isol-eq} and Condition \ref{cond:hurwitz-cdt-dcac} and follow ideas inspired from \cite{forni2013differential}. We take the derivative of the Lyapunov function \eqref{eq: LF} and add and substract $A(z^*)$ as defined in \eqref{eq: lin-sys}. Then, we obtain for all $(z, \delta z)\in\mathcal{D}$,
\begin{align*}
\dot V(\delta z)&=\delta z^\top\, \Pi \left( \frac{\partial f( z)}{\partial  z} \right) \delta z + \delta z^\top\,  \left(\frac{\partial f( z)}{\partial  z} \right)^\top\,\Pi\; \delta z, \\
& =\! \delta z^\top \, \left(P  A(z^*)+A(z^*)^\top P\right) \, \delta z+\delta z^\top \left(\Pi \, G(z) +G(z)^\top \, \Pi \right) \, \delta z, \\
& =\! -\delta z^\top \,Q(P) \delta z+\delta z^\top \left(\Pi \, G(z) +G(z)^\top \,\Pi \right) \, \delta z, 
\end{align*}
where $\delta z\in T_z\mathcal{D}$, and the matrix $G(z) = \frac{\partial f(z)}{\partial  z}-A(z^*)$ is given by,
\begin{align*}
G(z) =\left[\scalemath{0.8}{\begin{array}{c|c} G_{11} & G_{12}\\ \hline G_{21} & 0
\end{array}}\right]= \left[
\scalemath{0.7}{
\begin{array}{c c| c c c} 
0 & 0  & 0 & 0 & 0 \\   -C_{dc}^{-1}\widehat W(z) & 0 & -C_{dc}^{-1}\widehat Y(z)^\top & 0 & 0 \\ \hline L^{-1}\widehat M(z) & L^{-1}\widehat Y(z) & 0 & 0 & 0 \\ 0 & 0 & 0 & 0 &0\\ 0 & 0 & 0 & 0 &0
\end{array}}\right], \end{align*} and we define the matrices,
\begin{align*}
\widehat W( z)&=\frac{1}{2}\; \mu\; \text{diag}\left( (\mathbf{J} \mathrm{Rot}(\gamma))^\top i_{}-(\mathbf{J} \mathrm{Rot}( \gamma^*))^\top  i^*_{}\right), \\
\widehat Y( z)&=\frac{1}{2} \mu \,  \bigl(\mathrm{Rot}(\gamma)-\mathrm{Rot}(\gamma^*)\bigr), \ \mu\in[0,1] \\
\widehat M(z)&=\frac{1}{2} \mu\,   \, \bigl(\text{diag}(v_{dc}) \, \mathbf{J}\mathrm{Rot}(\gamma)-v^*_{dc}\,  \mathbf{J}\mathrm{Rot}(\gamma^*)\bigr).
\end{align*}

{By Conditions \ref{cond:hurwitz-cdt-dcac}, we have that $\text{span}\{v(z^*)\}\subset \mathcal{C}_\epsilon$ is asymptotically stable, which establishes that the set $\mathcal{C}_\epsilon$ is forward invariant.}

On the set $\mathcal{C}_\epsilon\subseteq \mathcal{D}$ given by \eqref{eq: roc}, we have that, for all $\gamma \in \Gamma(z,[z^*])$,
\[
d(z,[z^*])\leq \inf_{\Gamma(z,[z^*])}\int_\gamma \sqrt{\epsilon}\, ds< \epsilon',\epsilon'>0.
\]
which follows from the definition of Finsler distance in \eqref{eq: finsler-dist}.
This implies in particular that, $\vert\vert z-z^*\vert\vert<\epsilon'$, since $z^*\in[z^*]$, hence there exists sufficiently small $\hat \epsilon>0$, so that, 
$\delta z^\top \left (\Pi  \, G(z) +G(z)^\top \,\Pi  \right) \, \delta z\leq \hat\epsilon$. Thus, we have that 
$\dot V(\delta z) \leq  -\delta z^\top Q(P) \delta z+\hat\epsilon$, with $Q(P)$ given in \eqref{eq: LF}.

By choice of $Q_1=q_1\mathbf{I}$, $Q_2=q_2\frac{(\mathbf{J}x^*) (\mathbf{J}x^*)^\top}{(\mathbf{J}x^*)^\top (\mathbf{J}x^*)},\, q_1,q_2>0$ and from $\hat\epsilon\to 0$, we have that $\dot V(\delta z)=0 \Leftrightarrow \delta z=\text{span}\{v(z^*)\}$.

By \cite[Theorem 1]{forni2013differential}, the system \eqref{eq:var-sys} and hence \eqref{eq: non-lin} defined on the quotient space $\mathcal{D}/\sim$ is incrementally asymptotically stable on $\mathcal{C}_\epsilon/\sim$. As a consequence, the solutions to \eqref{eq: non-lin} are asymptotically contracting towards $[z^*]$ for all trajectories inititalized on $\mathcal{C}_\epsilon$ given by \eqref{eq:increm-stab}.
\end{proof}

\begin{figure}[h!]
	\vspace{-1.35cm}
	\includegraphics[width=0.58\linewidth, center]{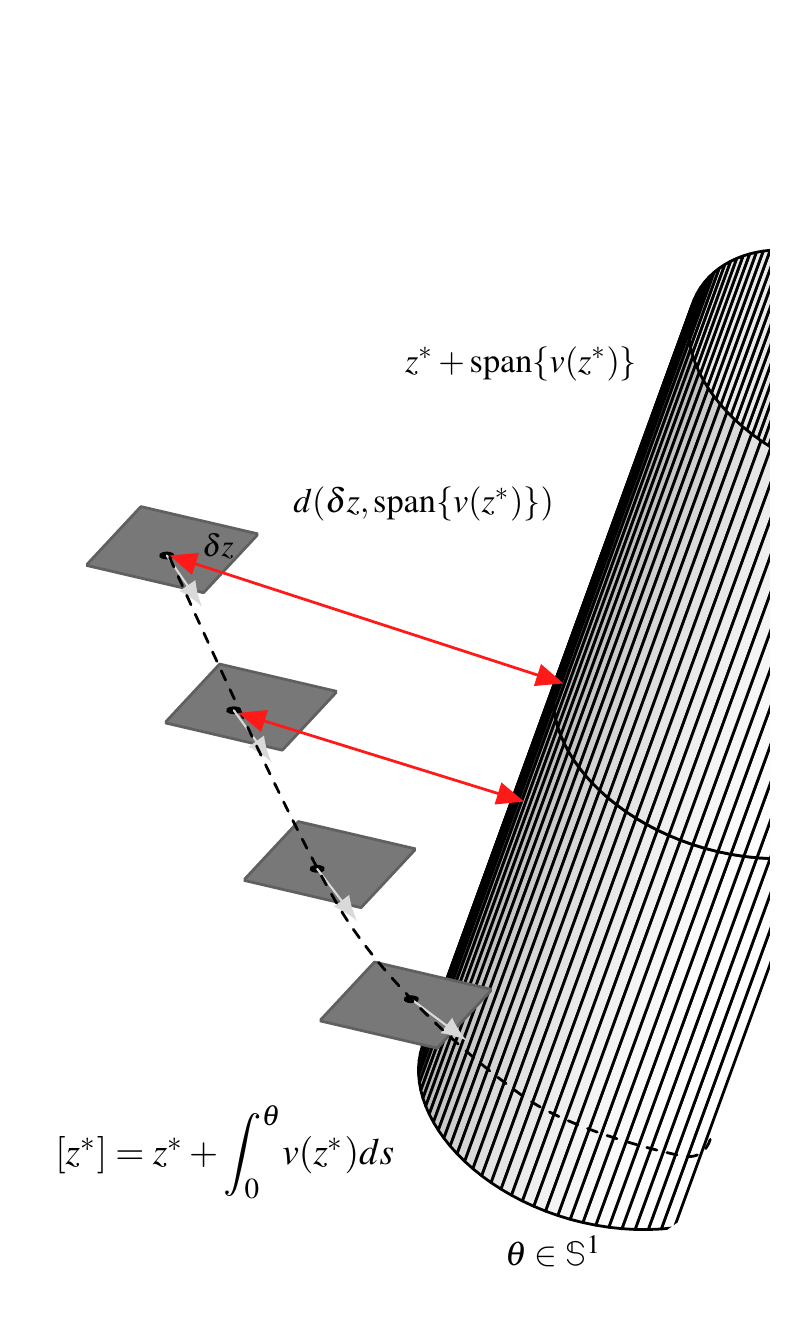}
	\caption{Convergence of a solution to \eqref{eq: non-lin} initialized on $\mathcal{C}_\epsilon\subseteq \mathcal{D}$ into the synchronous equilibrium manifold $[z^*]$ under Assumptions \ref{ass: feasibility},\ref{ass:isol-eq} and Condition \ref{cond:hurwitz-cdt-dcac}. The distance of the linearized trajectories $\delta z$ on the tangent space to the subspace $\mathrm{span}\{v(z^*)\}$ shrinks and corresponds to the contraction of the solution towards the equilibrium manifold. The lines on the surface represent the vector $z^*+\mathrm{span} \{v(z^*)\}$. Integrating over $\theta\in\mathbb{S}^1$ yields the equilibrium manifold $[z^*]$.}%
	\label{fig: proof}
\end{figure}

\subsection{Integral curve of $\mathrm{span}\{v(z^*)\}$}
We establish a formal link between the linear subspace $\mathrm{span}\{v(z^*)\}$ and the synchronous equilibrium manifold $[z^*]$, following our stability approach depicted in Figure~\ref{fig: proof}. In fact, the convergence of linearized trajectories on the tangent bundle to $\mathrm{span}\{v(z^*)\}$ corresponds to the convergence of nonlinear solutions to $[z^*]$ via integration. It hold that for $\theta\in\mathbb{S}^1$, 
\begin{align*}
[z^*]=z^*+\int_{0}^\theta v(z^*) \; \mathrm{d}s=z^*+{\mathlarger{\mathlarger{\int}}_{0}^\theta \begin{bmatrix}
\mathds{1}_n \\0 \\ \mathbf{J}\, \mathbf{R}(s)\, x^* 
\end{bmatrix} \; \mathrm{d}s},
\end{align*}
which follows from \eqref{eq: equi-class}. In fact, $v(z^*)$ is the tangent vector of $[z^*]$ in the $\theta-$ direction and lies on the tangent space $T_{z^{*}}\mathcal{M}$. This can also be deduced from \eqref{eq:invariance} by expanding the Taylor series around $(\theta',z^*),\, \theta'\in\mathbb{S}^1,\ z^*\in\mathcal{M}$ of left and right terms in \eqref{eq:invariance} and comparing the terms of their first derivatives with respect to $\theta$. In this way, we obtain, 
\begin{align*}
\left.\frac{\partial f(z)}{\partial z}\right \vert_{z=z^*} \left(\left.\frac{\partial H}{\partial \theta}\right \vert_{\theta'} \; z^*+h_0\right) \; (\theta-\theta') = 0
\end{align*}
where $\frac{\partial f(z)}{\partial z}\ \vert_{z=z^*}=A(z^*)$ and $h_0+\frac{\partial H}{\partial \theta}\vert_{\theta=\theta^{' }}\; z^*=v(z^*)$ (by the equivalence relation~\eqref{eq: equi-class}), hence we recover $A(z^*)\, v(z^*)=0$.

Theorem \ref{thm: main} specifies a parameterized forward invariant set representing the {\em contraction region}, see \cite{lohmiller1998contraction} for solutions of the power system model \eqref{eq: non-lin}, characterized by small distance of angles, DC voltages (and thus frequency), and AC signal to the subspace $\text{span}\{v(z^*)\}$ and hence to the set $[z^*]$. For similar conditions, considering reduced-order power system models, we refer the reader to phase cohesiveness in\cite[Theorem 4.1]{weiss2019stability} and frequency boundedness in \cite[Lemma 4.1]{zhu2018stability}.

\subsection{Equilibria of interest}

We are  interested in those equilibria $z^*\in\mathcal{M}$ that verify the steady state condition \eqref{eq: are-cdt}. This condition can be evaluated in a fully decentralized fashion and is dependent on the converter's resistance $R$, modulation amplitude $\mu\in[0,1]$, nominal DC voltage $v_{dc}^*$ and reactive power output $Q_{sw,k}^*$. In particular, condition \eqref{eq: are-cdt} requires sufficient reactive power support and resistive damping, which are well-known practical stability conditions \cite{wang2015virtual}. 
In addition to virtual impedance and current measurement (see \cite[Remark 2]{jouini2019parametric}), constant reactive power load, set in parallel with the load conductance $G>0$, can equivalently be considered to satisfy \eqref{eq: are-cdt}.

\subsection{Link to other stability theories}

Our stability analysis finds roots in concepts of partial contraction theory  \cite{russo2011symmetries, lohmiller1998contraction,wang2005partial} (or termed semi-contraction  \cite{lohmiller1998contraction}), allowing to extend the application of contraction analysis, to include convergence to {\em behaviors}, e.g., convergence to an equilibrium manifold. This can be interpreted as the contraction of the linearized trajectories in all directions up to that of the linear subspace $\mathrm{span}\{v(z^*)\}$, see \cite[Example 4.2]{wang2005partial}. In fact, the symmetric part of the Jacobian projected  into the orthogonal complement of $\mathrm{span}\{v(z^*)\}$), given by $-\left(\Pi \left(\frac{\partial f}{\partial z}\right)+\left(\frac{\partial f}{\partial z}\right)^\top\Pi\right)${\tb,} is positive definite with respect to $\mathrm{span}\{v(z^*)\}$, for trajectories initialized on  $\mathcal{C}_\epsilon$.

\section{Simulative example}
\label{sec: sims}
We consider two identical DC/AC converter model in closed-loop with the matching control  and connected via an RL line, as in \eqref{eq: multi-node}. The network setup and parameters can be found in the Table \ref{table_example}. 
\begin{table}[h!]
	\begin{center}
		\begin{tabular}{|l||l|l|l|}
			\hline
			& Converter $1$ & Converter $2$ & RL Line\\
			\hline\hline
			$i_{dc}^*$  &$37.23$ & $ 37.23$ & --\\
			$v_{dc}^*$ & $1000$ & $1000$& -- \\
			$C_{dc}$ &  $10^{-3} $& $10^{-3}$& --\\
			$G_{dc}$ & $10^{-5} $ & $10^{-5} $& --\\
			$\eta$ & $0.0003142$  & $0.0003142$ & -- \\
			$L$ & $5\cdot10^{-4}$ & $5\cdot10^{-4}$ &-- \\
			$C$ & $10^{-5}$ & $10^{-5}$ &-- \\
			$\mu^*$ & $0.33$ & $0.33$ &-- \\
			$G$ & $0.01$ & $0.01$ & -- \\  
			$b$ & $1.08$ & $1.08$ & -- \\  
			$R_{}$ & $0.2$&$0.2$ &-- \\
			$K_{p}$ & $0.099 $ & $0.099 $ & -- \\
			$R_{net}$ &--  &-- & $0.2$ \\
			$L_{net}$ & -- &-- & $5\cdot 10^{-5} $\\
			\hline
		\end{tabular}
	\end{center}
	\caption{Parameter values of the two DC/AC converters~(in p.u).}
	\label{table_example}
\end{table}

By choice of the current source $i_{dc}^*$, Assumption \ref{ass: feasibility} is verified.
Since the synchronous equilibrium satisfies the algebraic condition in \eqref{eq: are-cdt} after adding a reactive load $b=1.08$, we  numerically find an estimate of the region of contraction $\mathcal{C}_\epsilon$, defined by \eqref{eq: roc} in a systematic way following estimate in \eqref{eq: roc}. 

Figure \ref{fig: roc} depicts the region of contraction of the two DC/AC converter angles (in rad) and the convergence of angle solutions to the subspace $\mathds{1}_2$, for $\epsilon=3.5,$ resulting from varying the initial angles, while keeping the remaining initial states fixed and showing the convergence to equilibrium manifold as predicted by our theory. Hereby, we notice in particular the synchronization of DC voltages, and that AC signals remain close to their steady state values.
A large range of bounded disturbances (estimated by transient power values) can be considered in our simulations, despite the conservativeness of the estimate of the region of contraction. 
\begin{figure}
	\label{fig: roc}
	\centering
	\includegraphics[scale=0.3]{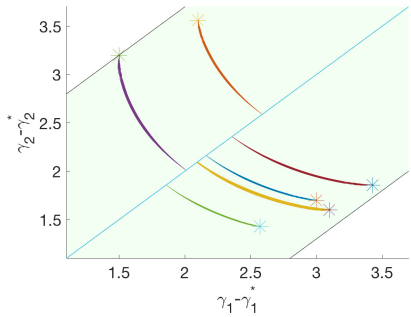}
	\caption{Region of contraction of the two-DC/AC converter angles (in rad) and convergence of the sample angle solutions of \eqref{eq: non-lin} to the subspace $\mathds{1}_2$, for $\epsilon=3.5$, and resulting from varying the initial angles, while keeping the remaining initial states fixed. The initial conditions of angle deviations $\gamma_1-\gamma_1^*$ and $\gamma_2-\gamma_2^*$ are denoted by the different stars and all the angles are in rad.}
	
\end{figure}

\section{CONCLUSIONS}
We considered local convergence of a multi-converter power system model. The symmetry of the vector field allowed for the adoption of a Lyapunov based framework with considerations in the quotient space. Our Lyapunov function is a distance measure from the  solution of the power system model to the synchronous equilibrium manifold that shrinks under sufficient and fully decentralized conditions, for trajectories initialized on region of the space characterized by small distance to the tangent space of the rotational invariance. Our numerical simulations validate our results. The scope of future
investigations includes extensive numerical estimations of the region of contraction and investigation of the conservativeness of our estimate.

\addtolength{\textheight}{-12cm}   




\bibliographystyle{IEEEtran}
\bibliography{root}

\end{document}